\documentclass[10pt,twoside]{article}
\usepackage{Latex-document}
\almvol{00}
\almttone{Frontiers of Science Awards}
\almtttwo{for Math/CS/Phys}
\firstpage{1}
\usepackage[english]{babel}
\usepackage[utf8]{inputenc}

\usepackage{multicol}

\usepackage{tabu}

\usepackage{amsmath}
\usepackage{amssymb}

\numberwithin{equation}{section}

\newtheorem{thm}{Theorem}[section]
\newtheorem{defi}[thm]{Definition}
\newtheorem{prop}[thm]{Proposition}
\newtheorem{lem}[thm]{Lemma}

\newtheorem{rem}[thm]{Remark}

\newtheorem{example}[thm]{Example}

\usepackage{tikz}
\usetikzlibrary{cd}
\usepackage{enumitem}
\setlist{topsep=0.1em,partopsep=0em,parsep=0em,itemsep=0.1em}

\newcommand\RR{\mathbb R}
\newcommand\CC{\mathbb C}
\newcommand\pRR{{}^\prime\mathbb R}

\newcommand\uSec{\mathrm{\underline{Sec}}}
\newcommand\uHol{\mathrm{\underline{Hol}}}
\newcommand\Hom{\mathrm{Hom}}
\newcommand\reg{\mathrm{reg}}
\newcommand\gp{\mathrm{gp}}
\newcommand\cM{\mathcal M}
\newcommand\cO{\mathcal O}
\newcommand\cP{\mathcal P}
\newcommand\cQ{\mathcal Q}
\newcommand\cZ{\mathcal Z}
\newcommand\sfC{\mathsf C}
\newcommand\sfD{\mathsf D}
\newcommand\sfE{\mathsf E}
\newcommand\sfK{\mathsf K}
\newcommand\sfP{\mathsf P}
\newcommand\op{\mathsf{op}}

\newcommand\Vect{\mathsf{Vect}}
\newcommand\Top{\mathsf{Top}}
\newcommand\Sm{\mathsf{Sm}}
\newcommand\Log{\mathsf{Log}}
\newcommand\Der{\mathsf{Der}}

\newcommand\Perf{\mathsf{Perf}}
\newcommand\Open{\mathsf{Open}}
\newcommand\Shv{\mathsf{Shv}}
\newcommand\opemb{\mathsf{opemb}}
\newcommand\coker{\operatorname{coker}}
\newcommand\abs[1]{\lvert#1\rvert}
\newcommand*\smashxrightarrow[1]{\mathrel{\mathstrut\smash{\raisebox{-.3ex}{$\xrightarrow{\smash{\raisebox{-.1ex}{$\scriptstyle #1$}}}$}}}}
\newenvironment{proof}{\noindent\emph{Proof:\ }}{\hfill$\blacksquare$\medskip}
\newlength\dummy

\begin{document}

\markboth{\hfill{\rm John Pardon} \hfill}{\hfill {\rm Representability in non-linear elliptic Fredholm analysis \hfill}}

\title{Representability in non-linear elliptic Fredholm analysis}

\author{John Pardon}

\begin{abstract}
We summarize current work aimed at showing that moduli spaces of solutions to non-linear elliptic Fredholm partial differential equations are derived log smooth manifolds.
\end{abstract}

\maketitle

\setcounter{tocdepth}{1}
\tableofcontents

\section{Introduction}

The motivating question for this research summary is the following:
\begin{center}
\parbox{0.8\textwidth}{\emph{What sort of mathematical object is the moduli space of solutions to a non-linear elliptic Fredholm partial differential equation?}}
\end{center}
At the most basic level, such moduli spaces are topological spaces (equipped with the topology of uniform convergence in all derivatives).
This topological structure is (in general) insufficient for enumerative questions (such as asking for a `signed count' of solutions).

It is a classical fact going back to Kuranishi \cite{kuranishi} and Atiyah--Hitchin-Singer \cite{atiyahhitchinsinger} that the moduli spaces above may be expressed locally as the zero set of a smooth map $\RR^n\to\RR^m$.
Such a chart is called a \emph{Kuranishi chart}, and captures enumerative information \emph{locally}.
The idea of equipping such a moduli space with an \emph{atlas} of Kuranishi charts and patching together their enumerative information globally first appeared in work of Fukaya--Ono \cite{fukayaono} and was developed further by Fukaya--Oh--Ohta--Ono \cite{foooI,foooII} and others.
The construction of such atlases, as well as the axioms they are required to satisfy, remains \emph{ad hoc}, despite many years of effort from a number of authors to establish a more canonical approach.
A more canonical approach would be highly desirable, as it could be expected to eliminate the need for explicit prescriptions for, and delicate manipulations of, atlases of compatible Kuranishi charts, thus enabling `operadic' reasoning about moduli spaces of pseudo-holomorphic curves to be independent of the foundational discussion (something which is not possible with current technology).
A prerequisite for such an approach is to understand more intrinsically what structure such an atlas of charts is really describing.

In algebraic geometry, there is an established approach to the definition and construction of moduli spaces based on \emph{moduli functors} (due to Grothendieck, Artin, Deligne--Mumford, and others).
The basic idea (trivial, yet revolutionary) is that to specify a moduli space $\cM$, it is equivalent to specify the functor of maps $Z\to\cM$ from spaces $Z$, and this is supposed to be `families of objects parameterized by $Z$' (the term `space' is really a placeholder: we could use any category we like here, such as topological spaces, smooth manifolds, complex analytic spaces, etc.).
The moduli functor
\begin{equation*}
Z\mapsto\{\text{families of objects parameterized by }Z\}
\end{equation*}
is usually much more straightforward and tautological to define than the moduli space itself.
Of course, one still needs to show that the moduli functor is \emph{representable} (i.e.\ is of the form $Z\mapsto\Hom(Z,\cM)$ for some space $\cM$).
Crucially, representability is a \emph{property} (the space $\cM$ is automatically unique up to unique isomorphism, if it exists), and moreover it is a \emph{local} property.
Suddenly we have gained something for free: local charts glue together automatically!
In a similar vein, building compatible atlases on a collection of related moduli spaces is quite delicate, whereas the corresponding moduli functors are related tautologically, hence so are their representing objects.

We may thus ask: what moduli functors can we associate to a non-linear elliptic Fredholm partial differential equation, and are they representable?
It is fairly straightforward to define a moduli functor on topological spaces and to show that it is represented by the topological moduli space alluded to above.
It is even easier to define the moduli functor on smooth manifolds, and standard non-linear elliptic Fredholm analysis shows that this functor is representable over the open locus where the linearized operator is surjective.
It is thus natural to ask whether there exists a reasonable moduli functor on a suitable category of `spaces with an atlas of compatible Kuranishi charts' and whether this functor is representable.
This was conjectured explicitly by Joyce \cite[\S 5.3]{joycenewkuranishi}, and it is the differential geometric analogue of the derived approach to algebraic Gromov--Witten theory proposed by Kontsevich \cite{kontsevich}, and developed by Ciocan-Fontanine--Kapranov \cite{cfkquot,cfkhilbert} and Kern--Mann--Manolache--Picciotto \cite{kernmannmanolachepicciotto}.

It is already quite nontrivial to define a suitable category of `spaces with an atlas of compatible Kuranishi charts'.
We can take a first hint from the case of linear elliptic operators.
The kernel of an elliptic operator $L:E\to F$ on a manifold $M$ is the fiber product of vector spaces $C^\infty(M,E)\times_{C^\infty(M,F)}0$.
On the other hand, if we instead take this same fiber product in the $\infty$-category $\sfK^{\geq 0}(\Vect_\RR)$ of complexes supported in non-negative cohomological degree, then we obtain the two-term complex $[C^\infty(M,E)\smashxrightarrow LC^\infty(M,F)]$, which is quasi-isomorphic to its cohomology $[\ker L\smashxrightarrow 0\coker L]$.
That is, passing to a `derived' setting exactly captures the full complex associated to $L$, rather than just its kernel.
Now for non-linear equations, we need a non-linear generalization of this setup.

The $\infty$-category of \emph{derived smooth manifolds} $\Der$ is a non-linear analogue of $\sfK^{\geq 0}(\Vect_\RR)$.
It may be obtained from the category of smooth manifolds $\Sm$ by formally adjoining finite limits modulo transverse limits, see Definition \ref{derdef} and the surrounding discussion.
Derived smooth manifolds were introduced by Spivak \cite{spivakthesis,spivak}, and fall within the rather general framework of derived geometry introduced by Lurie \cite{luriethesis} and Toën--Vezzosi \cite{toenvezzosihagi,toenvezzosihagii}.
From the very beginning, one motivation for this theory was to capture intersection multiplicities using fiber products (the connection goes all the way back to Serre's intersection formula \cite[V.C.1]{serremultiplicities}).
A derived smooth manifold $X$ has a tangent complex $TX\in\Perf^{\geq 0}(X)$ (locally a finite complex of vector bundles supported in cohomological degrees $\geq 0$), and $X$ is called \emph{quasi-smooth} when $TX$ is supported in degrees $[0\;1]$.
Quasi-smooth derived smooth manifolds have a reasonable bordism theory, which coincides with that of ordinary smooth manifolds, by Spivak \cite{spivak}.
In the analytic/algebraic setting, there is a well developed theory of virtual fundamental classes for quasi-smooth derived schemes / analytic spaces \cite{behrendfantechi,schurgtoenvezzosi,khanvirtual}.
The enumerative significance of arbitrary derived smooth manifolds is less clear (though the enumerative theory does extend at least somewhat beyond the quasi-smooth setting, see for example Borisov--Joyce \cite{borisovjoyce}).
Atlases of Kuranishi charts and quasi-smooth derived smooth manifolds are related by work of Joyce \cite{joycedman,joycenewkuranishi,joycekuranishi2cat}.

Our `main result' Theorem \ref{representabilitytheorem} (though we cannot exactly call it a `result' as we only have space for a brief sketch of the proof) is that moduli functors of pseudo-holomorphic maps from compact smooth Riemann surfaces (or families thereof) are representable on the $\infty$-category of derived smooth manifolds (an independent proof has been announced by Pelle Steffens \cite{steffensannounce,steffensforthcoming}).
Our proof consists of three main steps and, remarkably, reveals the result to be a formal categorical consequence of standard Fredholm analysis (the inverse/implicit function theorem for smooth Banach manifolds).
The first step is the `standard Fredholm analysis' to show that the `regular locus' (where the linearized operator is surjective) is representable as a stack on smooth manifolds.
The second step (which is the heart of the proof) is to deduce, formally, from this fact, that the regular locus remains representable on all \emph{derived} smooth manifolds, by the same representing object (smooth manifold).
The third step is a standard (and trivial) transverse thickening argument to deduce representability from representability of the regular locus.
We should point out that this proof  is not particularly specific to the setting of pseudo-holomorphic curves: similar reasoning should apply to any non-linear elliptic Fredholm problem with two-term deformation theory.
While we expect Theorem \ref{representabilitytheorem} to remain valid for problems with elliptic deformation complex of arbitrary length, the extension of our arguments to treat that case would be nontrivial.

While enumerative applications were a key motivation for Theorem \ref{representabilitytheorem}, it does not concern these as such.
Rather, it must be combined with a suitable theory of virtual fundamental cycles (or bordism) for (certain, e.g.\ quasi-smooth) derived smooth manifolds (which we view as a separate problem).

For most interesting applications, it is necessary to consider moduli spaces of pseudo-holomorphic maps from families of degenerating curves (which are not covered by Theorem \ref{representabilitytheorem}).
Joyce \cite{joyceanalyticcorners} has proposed the framework of (what we call) `log smooth manifolds' (with origins in work of Melrose \cite{melroseicm,melroseconormal,melroseindex}) for formulating and proving representability of moduli spaces in this setting (there is also closely related work of Parker \cite{parkerexploded}).
We discuss this briefly at the end.

\subsection{Acknowledgements}

The author is grateful for comments from the anonymous referee, Kenji Fukaya, Tobias Ekholm, André Henriques, Pelle Steffens, Dennis Sullivan, and Runjie Hu.

\section{Linear elliptic equations}

Given an elliptic operator $L:E\to F$ on a compact manifold $M$, the kernel and cokernel are finite-dimensional.
More generally, given an elliptic complex $E^0\smashxrightarrow{L_0}E^1\smashxrightarrow{L_1}\cdots\smashxrightarrow{L_{n-1}}E^n$ on a compact manifold, its cohomology groups are finite-dimensional.
Rather than taking cohomology, it is somewhat better to say that these complexes are isomorphic in the $\infty$-category of complexes to finite complexes of finite-dimensional vector spaces (though there is little difference in this simple setting).

Things get more interesting if we consider families of elliptic operators.
Let $L:E\to F$ be a vertical elliptic operator on a proper submersion $\pi:Q\to B$ of smooth manifolds.
Suppose $L$ is surjective on the fiber over $b\in B$.
It is then surjective on nearby fibers, and its kernel forms a smooth vector bundle over a neighborhood of $b$ in the base $B$.
What do we mean by this last statement?
The most straightforward interpretation is that there is a natural way to define local trivializations and that one can check that the transition maps between these are smooth.
But there is a better approach using representable functors.
We ask: what should a map from a smooth vector bundle $V$ over $B$ to the bundle $\ker L$ be?
The answer is obvious: it should be a smooth map $\pi^*V\to E$ annihilated by $L$.
Now we may ask: is this functor representable?
Notice that now representability is a property, and it suffices to prove it locally (things glue automatically since representing objects are unique up to unique isomorphism).
So, we never have to compare local trivializations, rather we just have to construct local charts satisfying a \emph{property}.
Over the open subset of the base where $L$ is surjective, representability of this functor follows from standard elliptic analysis, and the fiber of the representing object at $b\in B$ is indeed the kernel of $L$ acting on sections over the fiber of $Q\to B$ over $b$.

Now what happens in general, when $L$ is not assumed surjective?
We are now searching for a two-term complex of smooth vector bundles `up to homotopy equivalence'.
More precisely, we are looking for an object of the 2-category $\Perf^{[0\;1]}(B)$ whose objects are two-term complexes of vector bundles $[V^0\to V^1]$ on $B$, whose 1-morphisms are chain maps, and whose 2-morphisms are chain homotopies, and there are no higher morphisms for degree reasons (to be completely precise, this describes a presheaf of 2-categories on $B$, and $\Perf^{[0\;1]}$ is its sheafification).
We should now write down a functor on $\Perf^{[0\;1]}(B)$ whose representing object will be the `derived pushforward' $\pi_*L$ of $L$.
This functor sends $V^\bullet\in\Perf^{[0\;1]}(B)$ to (the groupoid of cycles in) the complex of global sections of $\Hom(\pi^*V,[E\smashxrightarrow LF])$, whose differential is the sum (with appropriate signs) of $L$ and the differential $V^0\to V^1$.
It can be checked, essentially by reducing to the surjective case, that this functor is representable and that the fiber of its representing object over a point $b\in B$ is indeed quasi-isomorphic to $[C^\infty(Q_b,E)\smashxrightarrow{L_b}C^\infty(Q_b,F)]$ (equivalently $[\ker L_b\smashxrightarrow 0\coker L_b]$).
Notice that the local representing objects (two-term complexes of finite-dimensional vector bundles on $B$) are only unique up to unique (up to unique 2-isomorphism) 1-isomorphism in the 2-category $\Perf^{[0\;1]}$, which means the gluing data involved in patching them together is quite a bit more complicated than in the fiberwise surjective setting.
The formalism of representable functors is thus a significant advantage in this case, as it allows us to construct and reason with the derived pushforward $\pi_*L$ without manipulating the patching data directly.

This discussion generalizes readily to families of elliptic complexes.
There is an $\infty$-category $\Perf^{\geq 0}$ whose objects are described locally as finite complexes of vector bundles supported in non-negative cohomological degree and whose morphisms are given by the space of cycles in the usual mapping complex.
One can write down the analogous functor, whose representing object is called the derived pushforward of the elliptic complex.

\section{Derived smooth manifolds}

The first step in generalizing from linear to non-linear elliptic equations is to find the non-linear analogue of the $\infty$-category $\sfK^{\geq 0}(\Vect)$ of complexes of real vector spaces supported in non-negative cohomological degree.
This is the $\infty$-category $\Der$ of \emph{derived smooth manifolds}, defined by Spivak \cite{spivakthesis,spivak} following ideas of Lurie \cite{luriethesis} and Toën--Vezzosi \cite{toenvezzosihagi,toenvezzosihagii}, and developed further by Borisov--Noel \cite{borisovnoel}, Behrend--Liao--Xu \cite{behrendliaoxu}, Carchedi--Steffens \cite{carchedisteffens}, Carchedi \cite{carchedi}, and Taroyan \cite{taroyan}.
Joyce has defined a 2-category of `d-manifolds' \cite{joycedman,joycenewkuranishi,joycekuranishi2cat} which is closely related (if not literally equivalent) to the full subcategory of $\Der$ spanned by quasi-smooth objects.

In contrast to the aforementioned references, we will take an axiomatic approach to derived smooth manifolds (see Definition \ref{derdef}), which we believe minimizes the amount of technical input needed to get the theory off the ground (in particular, we do not need the notion of a homotopy $C^\infty$-ring).
From our perspective, the $\infty$-category of derived smooth manifolds is obtained from the category of smooth manifolds $\Sm$ by \emph{formally adjoining finite limits} modulo \emph{preserving finite transverse limits} (within the realm of \emph{topological $\infty$-sites}).

\begin{defi}[Topological $\infty$-site]
A topological $\infty$-site is an $\infty$-category $\sfC$ along with a functor $\abs\cdot:\sfC\to\Top$, with the property that for every diagram of solid arrows
\begin{equation*}
\begin{tikzcd}
*\ar[r]\ar[d,"1"']&\sfC\ar[d,"\abs\cdot"]\\
\Delta^1\ar[r,"\opemb"]\ar[ru,dashed]&\Top
\end{tikzcd}
\end{equation*}
in which the bottom arrow is an open embedding in $\Top$, there exists a dotted lift which is cartesian in the sense of \textup{\cite[\S 2.4.1]{luriehtt}}.
\end{defi}

An arrow in a topological $\infty$-site $\sfC$ which is cartesian over an open embedding in $\Top$ is called an open embedding in $\sfC$.
There is an equivalence $(\sfC\downarrow^\opemb X)=\Open(\abs X)$ for every object $X\in\sfC$.
We can thus make sense of the sheaf property for presheaves on $\sfC$ using this equivalence (namely, require the pullback to $\Open(\abs X)$ to be a sheaf in the usual sense for every $X\in\sfC$).
A topological $\infty$-site is called \emph{subcanonical} when Yoneda presheaves are sheaves (in other words, when a morphism out of $X\in\sfC$ amounts to certain \emph{local} data on $X$).
A topological $\infty$-site is called \emph{perfect} when it is subcanonical and every sheaf which is locally representable (has a cover by open substacks which are all representable) is representable (in other words, $\sfC$ is perfect when the result of gluing together objects of $\sfC$ along open sets is again an object of $\sfC$).

\begin{example}
The categories of topological spaces \textup{($\Top$)}, smooth manifolds \textup{($\Sm$)} (not necessarily paracompact or Hausdorff), complex analytic spaces, and schemes, are all perfect topological sites.
Another example of a topological site is the category $\Vect\rtimes\Top$ whose objects are pairs $(X,V)$ where $X\in\Top$ and $V/X$ is a vector bundle and in which a morphism $(X,V)\to(X',V')$ is a continuous map $f:X\to X'$ and a linear map $V\to f^*V'$ (the functor $\abs\cdot:\Vect\rtimes\Top\to\Top$ sends $(X,V)\mapsto X$).
\end{example}

\begin{defi}[Topological functor]
Let $\sfC$ and $\sfD$ be topological $\infty$-sites.
A \emph{topological functor} $\sfC\to\sfD$ is a functor $f:\sfC\to\sfD$ preserving open embeddings, together with a natural transformation $\pi:\abs{f(\cdot)}_\sfD\to\abs\cdot_\sfC$ which sends open embeddings to pullbacks.
A topological functor $(f,\pi)$ is called \emph{strict} when $\pi$ is a natural isomorphism.
\end{defi}

\begin{example}
The functor $\Sm\to\Sm$ given by $X\mapsto TX$ is a topological functor.
The forgetful functor $\Sm\to\Top$ is a strict topological functor.
The forgetful functor $\Vect\rtimes\Top\to\Top$ sending $(X,V)\mapsto X$ is a strict topological functor.
There is also a topological functor $\Vect\rtimes\Top\to\Top$ sending $(X,V)$ to the total space of $V$.
\end{example}

Now we define the $\infty$-category of derived smooth manifolds by the following set of axioms.

\begin{defi}[Derived smooth manifold]\label{derdef}
{\hskip 0ex minus 0.7ex}The perfect topological $\infty$-site $\Der$ together with the strict topological functor $\Sm\to\Der$ is defined by the following axioms:
\begin{description}
\item$\bullet$\ The functor $\Sm\to\Der$ is fully faithful and preserves finite products.
\item$\bullet$\ $\Der$ has finite limits, and every object of $\Der$ is locally isomorphic to a finite limit of smooth manifolds.
\item$\bullet$\ The functor $\abs\cdot:\Der\to\Top$ preserves finite limits.
\item$\bullet$\ For any $N\in\Sm$, the functor $\Hom(-,N):\Der\to\Shv(-)^\op\rtimes\Top$ sends finite cosifted limits (equivalently, totalizations of truncated cosimplicial objects) to relative limits over $\Top$.
\end{description}
\end{defi}

Note that the axioms determine quite directly the space of morphisms between any pair of finite limits of smooth manifolds in $\Der$, hence they determine the entire topological $\infty$-site $\Der$.

\begin{thm}[Universal property of derived smooth manifolds]
{\hskip 0ex minus 0.3ex}For any perfect topological $\infty$-site $\sfE$ with finite limits, the $\infty$-category of topological functors $\Der\to\sfE$ preserving finite limits is equivalent, via restriction, to the $\infty$-category of topological functors $\Sm\to\sfE$ preserving finite products.
\end{thm}

A similar universal property was proven by Carchedi--Steffens \cite{carchedisteffens}.

A diagram of vector spaces is called \emph{transverse} when its limit in $\Vect$ is preserved by the inclusion $\Vect\hookrightarrow\sfK^{\geq 0}(\Vect)$.
A diagram of smooth manifolds is called \emph{transverse} when for every point of its topological limit, the induced diagram of tangent spaces is transverse.

\begin{lem}
A topological functor $\Sm\to\sfE$ preserves finite transverse limits iff it preserves finite products.
\end{lem}

As an example application of the universal property of $\Sm\to\Der$, we note that the tangent functor $T:\Sm\to\Sm$ preserves finite products, hence extends uniquely to a functor $T:\Der\to\Der$.
This functor is right adjoint to $(-\times\tau):\Der\to\Der$ where $\tau$ is the derived zero set (fiber product in $\Der$) of the function $x\mapsto x^2$ (we also call $\tau$ the `universal tangent vector').
That is, $\Hom(X\times\tau,Y)=\Hom(X,TY)$ for derived smooth manifolds $X$ and $Y$.

The tangent complex of a derived smooth manifold $M$ at a point $x\in M$ is an object of $\sfK^{\geq 0}(\Vect_\RR)$ with finite-dimensional cohomology.
This cohomology detects the local structure: $T_{X/Y}$ is supported in degree $d$ precisely when $X\to Y$ is locally modelled on a pullback of the $d$th diagonal of $\RR^k\to *$ (equivalently, the $(d-1)$th diagonal of $*\to\RR^k$), and every map of derived smooth manifolds factors locally into a composition of such maps.
A derived smooth manifold whose tangent complex is supported in degrees $[0\;1]$ is called \emph{quasi-smooth}; for example, the derived zero set $f^{-1}(0)$ of a smooth function $f:\RR^n\to\RR^m$ (a Kuranishi chart) is quasi-smooth.

\section{Moduli stacks of pseudo-holomorphic maps}\label{modulistackssection}

Let us now make precise what we mean by a \emph{pseudo-holomorphic moduli problem} $\wp$ and its associated \emph{moduli stack} $\uHol(\wp)$ of solutions.
A quick note on terminology: we use the term `$\sfC$-stack' to mean any object of the $\infty$-category $\Shv(\sfC)$ of all sheaves (of spaces, aka $\infty$-groupoids) on the topological $\infty$-site $\sfC$ (for example topological stacks $\Shv(\Top)$, smooth stacks $\Shv(\Sm)$, and derived smooth stacks $\Shv(\Der)$).

At the simplest level, we could consider pseudo-holomorphic maps from a compact Riemann surface $C$ to an almost complex manifold $X$.
The moduli problem $\wp$ is the pair $(C,X)$, and its \emph{solutions} are pseudo-holomorphic maps $u:C\to X$.
In many contexts it is useful to generalize from maps to \emph{sections}.
This means we fix a pseudo-holomorphic submersion $W\to C$ over a compact Riemann surface $C$, and we consider solutions to be pseudo-holomorphic sections $u:C\to W$.

The \emph{moduli stack} $\uHol(\wp)$ of solutions to a pseudo-holomorphic moduli problem $\wp$ associates to each `space' $Z$ the set of `families' of solutions of $\wp$ parameterized by $Z$.
There are various sorts of moduli stacks depending on what we allow $Z$ to be, which we indicate by adorning $\uHol(\wp)$ with a subscript.

The simplest moduli stack to define is the \emph{smooth moduli stack} $\uHol(\wp)_\Sm$, which is when $Z$ is a smooth manifold.
A family of solutions of the moduli problem $\wp=(C,X)$ parameterized by $Z$ is simply a smooth map $Z\times C\to X$ whose restriction to each slice $z\times C$ is pseudo-holomorphic.
When $\wp$ is a section problem $(W\to C)$, a map $Z\to\uHol(\wp)_\Sm$ is then a smooth lift
\begin{equation*}
\begin{tikzcd}
{}&W\ar[d]\\
Z\times C\ar[r,"\pi_C"]\ar[ru]&C
\end{tikzcd}
\end{equation*}
whose restriction to each slice $z\times C$ is pseudo-holomorphic.

This object $\uHol(\wp)_\Sm$ is a \emph{smooth stack}, by which we mean an object of $\Shv(\Sm)$, the category of sheaves on $\Sm$, the category of smooth manifolds.
For a given smooth stack, we can ask whether it is \emph{representable}, meaning isomorphic in $\Shv(\Sm)$ to a functor of the form $\Hom(-,A)$ for some smooth manifold $A$.
Representability of $\uHol(\wp)_\Sm$ means, concretely, that there exists a smooth manifold $A$ and a \emph{universal family} $A\to\uHol(\wp)_\Sm$ (i.e.\ a family of pseudo-holomorphic maps/sections parameterized by $A$) such that for every smooth manifold $Z$, a family of pseudo-holomorphic maps/sections parameterized by $Z$ is the pullback of the universal family by a \emph{unique} map $Z\to A$.
It is not hard to check that such a universal object and universal family is unique up to unique isomorphism \emph{if it exists}.
If it exists, then it certainly deserves to be called the moduli space of pseudo-holomorphic maps/sections.

Standard non-linear elliptic Fredholm analysis shows that the smooth moduli stack is representable \emph{over the open set where the linearized operator is surjective} (let us call this the `regular locus' $\uHol(\wp)^\reg\subseteq\uHol(\wp)$, which is an open substack since surjectivity is an open condition for Fredholm operators).
To satisfactorily describe the entire moduli space (not just its regular locus), we need to introduce more complicated moduli stacks.

Let us next explain the \emph{topological moduli stack} $\uHol(\wp)_\Top$.
In other words, we should explain what it means to have a family of solutions to a pseudo-holomorphic moduli problem $\wp$ parameterized by a topological space $Z$.
The answer takes the same basic form as when $Z$ is a smooth manifold, namely we ask for a map $Z\times C\to X$ or to $W$.
At a minimum, we would certainly want this map should be smooth on the slices $z\times C$ and jointly continuous.
In fact, we want that all derivatives in the $C$ direction should exist and be jointly continuous.
These conditions can be stated succinctly by introducing the category $\Top\Sm$ of `topological-smooth spaces'.
Given a topological space $Z$ and an integer $n$, the (formal) product $Z\times\RR^n$ is an example of a topological-smooth space.
A continuous-smooth map $Z\times\RR^n\to Z'$ (possibly defined on just an open set) is a continuous map which is locally constant on every slice $z\times\RR^n$.
A continuous-smooth map $Z\times\RR^n\to\RR$ is a map whose derivatives to all orders in the $\RR^n$ direction exist and are jointly continuous.
Now the category $\Top\Sm$ is defined by taking atlases (an object is a topological space with an atlas of charts from open subsets of various $Z\times\RR^n$ with continuous-smooth transition maps, and a morphism is a continuous map which in every chart is continuous-smooth).
The topological moduli stack $\uHol(\wp)_\Top$ is (under mild conditions) representable.
Concretely, it is the set of pseudo-holomorphic maps/sections equipped with the topology of smooth convergence on compact subsets of the domain.

Now the moduli space which gives a truly satisfactory answer to our motivating question is the \emph{derived smooth moduli stack} $\uHol(\wp)_\Der$, which classifies families of pseudo-holomorphic maps parameterized by derived smooth manifolds.
To define this moduli stack, we just need to say when a morphism of derived smooth manifolds $Z\times C\to X$ is pseudo-holomorphic in the $C$ direction.
The point is that there is a tangent functor $T:\Der\to\Der$, which turns such a map into a map $TZ\times TC\to TX$.
We can now restrict it to $Z\times TC\to TX$, which is thus a section of the vector bundle $\Hom(TC,u^*TX)$ over $Z\times C$, and we can take the $(0,1)$-part and require it to vanish (which we should note is extra data, not a property, in this higher categorical context).
Note that this is \emph{not} a `fiberwise' constraint: there are functions $Z\to\RR$ which are nonzero yet whose pullback under every map $*\to Z$ is zero.
For example, a pseudo-holomorphic map $\tau\times C\to X$ (where $\tau$ is the universal tangent vector) is a pseudo-holomorphic map $C\to X$ along with first order deformation preserving pseudo-holomorphicity.
We emphasize that the definition of $\uHol(\wp)_\Der$ is entirely `synthetic'/`diagrammatic': all we need is the $\infty$-category $\Der\supseteq\Sm$ and its tangent functor, and some compatibilities with the notion of tangent space of smooth manifolds (at no point in defining this functor do we need to describe explicitly a morphism of derived smooth manifolds).

The above discussion generalizes quite easily to \emph{parameterized} moduli problems (which is the sort which usually appears in practice).
We consider a base `parameter' space $B$ together with submersions $W\to C\to B$ where $T_{W/B}$ and $T_{C/B}$ have complex structure, the map $W\to C$ is almost complex relative $B$, and $C\to B$ is proper of relative dimension two.
A map from $Z$ to the moduli stack $\uHol_B(C,W)$ is now a diagram
\begin{equation*}
\begin{tikzcd}
{}&W\ar[d]\\
C\times_BZ\ar[r]\ar[d]\ar[ru]&C\ar[d]\\
Z\ar[r]&B
\end{tikzcd}
\end{equation*}
where the diagonal map $C\times_BZ\to W$ is pseudo-holomorphic in the $T_{C/B}$ direction.
This `diagrammatic' definition works to define the topological, smooth, and derived smooth moduli stacks as above.
It is important to note here that we may take $B$ to be any topological, smooth, or derived smooth \emph{stack}: this generality is needed to describe most moduli spaces of interest, as in the following examples.

\begin{example}
A parameterized moduli problem over parameter space $B=*$ is the same as a moduli problem in the initial (unparameterized) sense.
\end{example}

\begin{example}
A parameterized moduli problem over parameter stack $B=*/G$ for a Lie group $G$ is necessarily of the form $W/G\to C/G\to */G$ where $W\to C$ is a $G$-equivariant pseudo-holomorphic section problem.
The moduli stack of the parameterized problem $\uHol_{*\!/\!G}(C/G,W/G)$ is the quotient $\uHol(C,W)/G$.
\end{example}

\begin{example}
Let $C\to B$ be the universal family over the moduli stack $B$ of compact smooth Riemann surfaces.
The moduli stack $\uHol_B(C,X)$ classifies (families of) `maps from compact smooth Riemann surfaces to $X$ modulo reparameterization'.
\end{example}

It is possible to identify the `tangent space' of a moduli stack of pseudo-holomorphic curves by formal reasoning at the level of moduli functors.
The tangent functor $T:\Der\to\Der$ induces a sheaf left Kan extension functor $T_!:\Shv(\Der)\to\Shv(\Der)$.
Now $T_!$ is alternatively the pullback $(-\times\tau)^*:\Shv(\Der)\to\Shv(\Der)$ (indeed, they agree on $\Der$ and are both cocontinuous, where $\tau$ denotes the derived zero set of $x^2:\RR\to\RR$).
Now the result of applying $(-\times\tau)^*$ to the moduli stack $\uHol(\wp)_\Der$ of $\wp=(W\to C\to B)$ may be identified quite directly with the moduli stack $\uHol(T\wp)_\Der$ associated to a certain `tangent moduli problem' $T\wp=(T_{W/C}\times_BTB\to C\times_BTB\to TB)$ depending on a choice of connection on $W\to B$.
Now the map $\uHol(T\wp)\to\uHol(\wp)$ is just a `relative' linear elliptic moduli problem.
We have thus shown that, for essentially formal reasons, the tangent space to the moduli stack of pseudo-holomorphic curves is given by the associated family of linear elliptic operators obtained by linearizing in the usual way.
It is remarkable that we can formulate and prove this statement \emph{before} we show that the moduli stack itself is actually representable!

\section{Representability}

We can now formulate and sketch the proof of our main `result' (an independent proof of which has been announced by Pelle Steffens \cite{steffensannounce,steffensforthcoming}).
It depends on two key results which we state and discuss afterwards.

\begin{thm}[Derived Regularity Theorem]\label{representabilitytheorem}
Let $W\to C\to B$ be a pseudo-holomorphic section problem over a derived smooth stack $B$, meaning $W\to C\to B$ are submersions in $\Shv(\Der)$, the map $B\to C$ is proper of relative dimension two, and $W\to C$ is pseudo-holomorphic with respect to specified complex structures on $T_{C/B}$ and $T_{W/B}$.
In this case, the map
\begin{equation*}
\uHol_B(C,W)\to B
\end{equation*}
is representable, and the tautological comparison map
\begin{equation*}
(\Der\to\Top)_!\uHol_B(C,W)_\Der\to\uHol_B(C,W)_\Top
\end{equation*}
is an isomorphism.
\end{thm}

\begin{proof}
We begin with some formal reasoning to reduce to the case that $B$ is a smooth manifold.
Formation of moduli stacks is compatible with pullback, so representability of $\uHol_B(C,W)\to B$ reduces immediately to the case that $B$ is a derived smooth manifold.
Given representability, formation of the comparison map is also compatible with pullback by Lemma \ref{rkanpullbacksheaf}, hence also reduces to the case $B$ is a derived smooth manifold.
Using Proposition \ref{leftkandercomparison}, one can show that every pseudo-holomorphic moduli problem over a derived smooth manifold $B$ is locally pulled back from a smooth manifold.
It thus suffices to consider the case $B$ is a smooth manifold.

We first claim that $\uHol_B(C,W)_\Sm^\reg$ is representable and that the comparison map $(\Sm\to\Top)_!\uHol_B(C,W)_\Sm^\reg\to\uHol_B(C,W)_\Top^\reg$ is an isomorphism (recall $\uHol^\reg\subseteq\uHol$ denotes the open substack where the linearized operator is surjective).
This is an application of standard non-linear elliptic Fredholm analysis using Newton--Picard iteration (formally the same as inverse function theorem).
More precisely, Newton--Picard iteration shows that certain natural `linear projections' $\lambda:\uHol_B(C,W)\to\RR^k$ are local isomorphisms on $\uHol_B(C,W)^\reg_\Top$.
It is not difficult to show that the local inverse $\RR^k\to\uHol_B(C,W)_\Top^\reg$ is continuously differentiable, hence that the linear projection is an isomorphism of stacks on $C^1$-manifolds.
Smoothness may be obtained formally by induction, by considering the `tangent moduli problem' (which is another elliptic partial differential equation).

Our second (and most significant) step is to show that, as a formal consequence of the first step (representability of $\uHol_B(C,W)_\Sm^\reg$), the comparison map $(\Sm\to\Der)_!\uHol_B(C,W)_\Sm^\reg\to\uHol_B(C,W)_\Der^\reg$ is an isomorphism (that is, families of regular pseudo-holomorphic sections over derived smooth manifolds are classified by the same smooth manifold classifying such families over smooth manifolds).
The underlying engine behind this fact is Proposition \ref{leftkandercomparison}, which says that the analogous assertion holds for the stacks of all sections.
To deduce the result for $\uHol_B(C,W)$, observe that $\uHol_B(C,W)$ is a fiber of $\uSec_B(C,W)\to\uSec(C_0,H_0)$ where $\uSec$ (resp.\ $\uSec_B$) is the stack of (resp.\ parameterized) smooth sections (no pseudo-holomorphicity imposed) and we have smoothly trivialized $C=C_0\times B\to B$ via Ehresmann and identified $\overline{T^*C}\otimes T_{W/C}$ with the pullback of a vector bundle $H_0/C_0$.
This map $\uSec_B(C,W)\to\uSec(C_0,H_0)$ is a submersion over its fiberwise regular locus by the previous paragraph.
Now we appeal to the fact that submersive pullbacks are preserved by left Kan extension $(\Sm\to\Der)_!$ by Proposition \ref{rkanpullbacksheaf}.

Having proven the result over the regular locus, we can deduce it everywhere using a standard thickening argument.
We already noted in the first paragraph of the proof that our desired conclusion is preserved under pullback.
Thus if our moduli problem $(W\to C\to B)$ is the pullback of another moduli problem $(\tilde W\to\tilde C\to\tilde B)$ under a map $B\to\tilde B$, then our desired result for $\uHol_B(C,W)$ holds over the open substack $\uHol_{\tilde B}(\tilde C,\tilde W)^\reg\times_{\tilde B}B\subseteq\uHol_B(C,W)$.
Now we just need to argue that $\uHol_B(C,W)$ may be covered by open substacks of this form, which is easy (take product of $(W\to C\to B)$ with $\RR^k$ and modify the holomorphic structure as a function of the $\RR^k$-coordinate, keeping it fixed at zero, so as to make any desired point of $\uHol_B(C,W)=\uHol_{\tilde B}(\tilde C,\tilde W)\times_{\RR^k}0$ regular inside $\uHol_{\tilde B}(\tilde C,\tilde W)$).
\end{proof}

The following result provides a powerful way to reduce statements about derived smooth manifolds to statements about smooth manifolds.

\begin{prop}\label{leftkandercomparison}
The comparison map
\begin{equation*}
(\Sm\to\Der)_!\uSec_B(C,W)_\Sm\to\uSec_B(C,W)_\Der
\end{equation*}
is an isomorphism for $C\to B$ proper.
\end{prop}

\begin{proof}
We begin this sketch with the case $B=*$.

We claim that a derived smooth stack lies in the essential image of left Kan extension $(\Sm\to\Der)_!:\Shv(\Sm)\to\Shv(\Der)$ iff its associated functor $\Der\to\Shv(-)^\op\rtimes\Top$ sends finite cosifted limits (equivalently, totalizations of truncated cosimplicial objects) to relative limits over $\Top$ (which involves \emph{co}limits of sheaves).
The final axiom of the $\infty$-category of derived smooth manifolds (Definition \ref{derdef}) implies $\Sm\subseteq\Shv(\Der)$ satisfies this condition.
Satisfaction of the condition is evidently closed under taking colimits, so everything in $\Shv(\Sm)\subseteq\Shv(\Der)$ satisfies it as well.
To prove the converse, it is enough (by the adjunction of $(\Sm\to\Der)_!$ and $(\Sm\to\Der)^*$) to note that if $F,G\in\Shv(\Der)$ satisfy the condition and a map $F\to G$ is an isomorphism over $\Sm$, then it is an isomorphism (since every derived smooth manifold is locally a finite limit of smooth manifolds).

Now, let us check that $\uSec(C,W)$ satisfies this criterion, i.e.\ that it sends finite cosifted limits (equivalently, totalizations of truncated cosimplicial objects) to relative limits over $\Top$.
Let $Q=\lim_\alpha Q_\alpha$ be a finite cosifted limit in $\Der$.
The final axiom of Definition \ref{derdef} implies that the functor of sections of $W$ sends the finite cosifted limit $Q\times C=\lim_\alpha Q_\alpha\times C$ to a relative limit over $\Top$.
The desired result for $\uSec(C,W)$ and $Q=\lim_\alpha Q_\alpha$ is the assertion that this relative limit diagram remains a relative limit diagram after pushing forward to $Q=\lim_\alpha Q_\alpha$.
By proper base change \cite[7.3.1.18]{luriehtt} (which applies since $C$ is compact Hausdorff), we reduce to the assertion that pushforward $\Shv(Q\times C)\to\Shv(Q)$ preserves a certain colimit diagram.
Proper pushforward preserves \emph{filtered} colimits by proper base change.
Proper pushforward does not preserve all pushouts, but it does in this case since the sheaves in question are \emph{soft} (have partitions of unity---the desired result is local, so we may assume wlog that $W\to C$ is a vector bundle).
This usage of softness (i.e.\ the existence of partitions of unity, since we are in the smooth setting) is essential, as the result in question fails in the analytic setting (under the same hypothesis that $C$ is compact Hausdorff).

To treat the case of general smooth manifolds $B$, it is easier to prove a stronger result, namely that $\uSec_B(C,W)\in\Shv(\Der\downarrow B)$ is in the essential image of $\Shv(\Sm\downarrow^{\mathsf{subm}}B)$ (sheaves on smooth manifolds submersive over $B$).
We may then follow the same strategy.
\end{proof}

The following result is pure category theory, though it is not quite trivial.

\begin{lem}\label{rkanpullbacksheaf}
Let $f:\sfC\to\sfD$ be a functor.
Let $\cP$ and $\cQ$ be properties of morphisms in $\sfC$ and $\sfD$ (respectively) preserved under pullback.
If $f$ sends pullbacks of $\cP$-morphisms to pullbacks of $\cQ$-morphisms, then so does the left Kan extension functor $f_!:\sfP(\sfC)\to\sfP(\sfD)$.
When $\sfC$ and $\sfD$ are perfect topological $\infty$-sites and $f$ is topological, the same holds for the sheaf left Kan extension functor $f_!:\Shv(\sfC)\to\Shv(\sfD)$.
\end{lem}

\section{Log smooth manifolds}

So far, we have only discussed pseudo-holomorphic maps from compact \emph{smooth} Riemann surfaces (and families thereof).
We now seek to generalize our discussion to Riemann surfaces with cylindrical ends and degenerating families thereof (as is necessary for most practical applications of the theory).
To do this, we need to fix a differential geometric context in which we can speak about such objects and define and prove representability of moduli functors of the same basic form considered in Section \ref{modulistackssection}.

Joyce \cite{joyceanalyticcorners} has proposed that the formalism of what we shall call \emph{log smooth manifolds} provides such a suitable differential geometric context.
The beginning of this theory is the `$b$-differential calculus' of Melrose \cite{melroseicm,melroseconormal,melroseindex} and its applications to linear elliptic equations.
The key notion of `log smoothness' seems to have been formalized first in work of Joyce \cite{joyceanalyticcorners} and, in a somewhat different form, Parker \cite{parkerexploded}, both of whom noted its applicability to pseudo-holomorphic curve problems.

Let us now define log smooth manifolds.
Given a \emph{real polyhedral cone} $P$ (an intersection of half-spaces in a finite-dimensional real vector space), there is a corresponding \emph{real affine toric variety} $X_P=\Hom((P,+),(\RR_{\geq 0},\cdot))$.
This $X_P$ is naturally stratified by the faces of $P$ (associate to a homomorphism $f:P\to\RR_{\geq 0}$ the face $f^{-1}(\RR_{>0})\subseteq P$), and it is known that $X_P$ and $P$ are homeomorphic as stratified topological spaces \cite[Theorem 1.4]{nakayamaogus}.
A \emph{log structure} on a topological space $X$ is a sheaf of monoids $\cO_X^{\geq 0}$ on $X$ equipped with a map to the sheaf of monoids $C_X^{\geq 0}$ of $\RR_{\geq 0}$-valued continuous functions under multiplication, with the property that this map is an isomorphism over the submonoid $C_X^{>0}\subseteq C_X^{\geq 0}$ of non-vanishing functions.
A map of log topological spaces $(f,f^\flat):(X,\cO_X^{\geq 0})\to(Y,\cO_Y^{\geq 0})$ is a continuous map $f$ along with a map of sheaves of monoids $f^\flat:\cO_Y^{\geq 0}\to f_*\cO_X^{\geq 0}$ compatible with pullback of $\RR_{\geq 0}$-valued functions.
We equip the space $X_P$ with the log structure associated to the tautological `pre-log structure' $P\to C_{X_P}^{\geq 0}$.
When $P=\RR_{\geq 0}$, we denote $X_P$ by $\pRR_{\geq 0}=X_{\RR_{\geq 0}}$.

\begin{defi}
To any element $p\in P$ there is an associated one-form on $X_P$ given by the pullback of $\frac{dx}x$ under the `evaluate at $p$' map $X_P\to\RR_{\geq 0}$ (dually, associated to an element of $(P^\gp)^*$ is a vector field on $X_P$).
This gives a notion of $C^1$-functions $X_P\to\RR$, and to each such function there is an associated derivative map $X_P\to P^\gp$.
We can inductively define $C^k$-functions $X_P\to\RR$ for all $k$, and so smooth functions as well.
A log map $X_P\to X_Q$ is called smooth when it is locally the product of a monomial map $X_P\to X_Q$ (i.e.\ induced by a map of polyhedral cones $Q\to P$) and a smooth map $X_P\to X_Q^\circ=(Q^\gp)^*$.
\end{defi}

Given this notion of differentiability, a log smooth manifold is a log topological space equipped with an atlas of charts from open subsets of various $X_P$, whose transition maps are log smooth.
A log smooth manifold locally modelled on $\pRR_{\geq 0}\times\RR^k$ may be reasonably called a ``manifold with asymptotically cylindrical ends'', as can be seen from the following example.

\begin{example}\label{logsmoothexample}
What are the log smooth maps $\RR^n\times\pRR_{\geq 0}\to\RR$ and $\RR^n\times\pRR_{\geq 0}\to\pRR_{\geq 0}$?
The answer to this question is easiest to grasp if we use \emph{log coordinates} $\pRR_{\geq 0}=(-\infty,\infty]$ identifying $s\in(-\infty,\infty]$ with $e^{-s}\in\pRR_{\geq 0}$.
In these coordinates, the log smooth maps $f:\RR^n\times(-\infty,\infty]=\RR^n\times\pRR_{\geq 0}\to\RR$ and $g:\RR^n\times(-\infty,\infty]=\RR^n\times\pRR_{\geq 0}\to\pRR_{\geq 0}=(-\infty,\infty]$ are precisely those maps with the following behavior as $s\to\infty$:
\begin{align*}
f(x_1,\ldots,x_n,s)&=f_\infty(x_1,\ldots,x_n)+o(1)_{C^\infty}\\
g(x_1,\ldots,x_n,s)&=g_\infty(x_1,\ldots,x_n)+a\cdot s+o(1)_{C^\infty}
\end{align*}
for smooth $f_\infty,g_\infty:\RR^n\to\RR$ and constant $a\geq 0$, where $o(1)_{C^\infty}$ denotes a function all of whose derivatives $\partial_s^\ell\partial_{x_1}^{m_1}\cdots\partial_{x_n}^{m_n}$ are $o(1)$ (decay to zero) as $s\to\infty$ (uniformly over compact subsets of $\RR^n$).

For applications to elliptic problems, one can impose a stronger \emph{exponential decay} condition, namely replace $o(1)_{C^\infty}$ with `$O(e^{-\delta s})_{C^\infty}$ for some (unspecified) $\delta>0$' (this is the distinction between `roughly smooth' and `analytically smooth' in Joyce's terminology).
\end{example}

\begin{rem}
Suppose $\varphi:\pRR_{\geq 0}\to\RR_{\geq 0}$ (often called a `gluing profile') is a homeomorphism with the property that for all log smooth $F:\pRR_{\geq 0}^k\to\pRR_{\geq 0}$ with exponential decay in the sense of Example \ref{logsmoothexample}, the conjugation $\varphi\circ F\circ(\varphi^k)^{-1}:\RR_{\geq 0}^k\to\RR_{\geq 0}$ is smooth (in the usual sense).
For example, $\varphi(t)=(-\log t)^{-1}=s^{-1}$ has this property (for $k=1$, this amounts to showing that $(r^{-1}+f(r^{-1}))^{-1}=r/(1+rf(r^{-1}))$ is smooth at $r=0$ whenever $f(s)=O(e^{-\varepsilon s})_{C^\infty}$ as $s\to\infty$, which follows from explicit differentiation).
Such a function $\varphi$ determines a functor from the category of log smooth manifolds `with exponential decay' locally isomorphic to $\pRR_{\geq 0}^k$ and log smooth maps with exponential decay, to the category of smooth manifolds with corners (by definition locally isomorphic to $\RR_{\geq 0}^k$) and smooth maps.
This functor does not respect tangent bundles on the nose, rather only `up to homotopy'.
\end{rem}

\begin{example}\label{puncturelog}
A punctured Riemann surface has a unique bordification to a log smooth manifold (since holomorphic maps $D^2\setminus 0\to D^2\setminus 0$, written in cylindrical coordinates $z=e^{s+it}$ for $(s,t)\in(0,\infty)\times S^1$, have the required form from Example \ref{logsmoothexample}).
\end{example}

\begin{example}\label{contactlog}
Given a co-oriented contact manifold $(Y,\xi)$, its symplectization $SY$ is the subspace of positive contact forms inside $T^*Y$, and it is equipped with the restriction of the tautological 1-form on $T^*Y$.
Given a positive contact form $\alpha$ on $Y$, the symplectization $SY$ has coordiates $\RR\times Y$ with the 1-form $e^s\alpha$.
The bordification $\overline SY=[-\infty,\infty]\times Y$ is naturally a log smooth manifold, obtained by gluing $s\in\RR$ to $e^s\in\pRR_{\geq 0}$ (near $s=-\infty$) and $e^{-s}\in\pRR_{\geq 0}$ (near $s=\infty$).
This bordification $\overline SY$ is independent of the choice of contact form, since the relevant coordinate change $(s,p)\mapsto(s+\frac{\alpha'}\alpha(p),p):\RR\times Y\to\RR\times Y$ is log smooth (compare Example \ref{logsmoothexample}).
\end{example}

\begin{example}
Hofer \cite{hoferweinstein} and Hofer--Wysocki--Zehnder \cite{hwzasymptotics} show that for any cylindrical almost complex structure on $SY$, every pseudo-holomorphic map $u:D^2\setminus 0\to SY$ with finite Hofer energy satisfies exponential decay in cylindrical coordinates in the sense of sense of Example \ref{logsmoothexample}, hence extends to a map of log smooth bordifications from Examples \ref{puncturelog} and \ref{contactlog}.
\end{example}

At a point $x$ of a log smooth manifold $M$, there is a natural short exact sequence
\begin{equation*}
0\to T^*_xM_x\to T^*_xM\to\cZ_{M,x}^\gp\to 0
\end{equation*}
where $M_x\subseteq M$ is the local stratum of $M$ containing $x$, and $\cZ_{M,x}$ is the sharp (meaning it contains no non-zero invertible elements) polyhedral cone govering the local structure of $M$ at $x$.
This short exact sequence is functorial in log smooth maps $f:M\to N$.

We call a map $f:M\to N$ a \emph{broken submersion} at $p\in M$ when $T^*_pf:T^*_{f(p)}N\to T^*_pM$ is injective and $f^{\flat\flat}_p:\cZ_{N,f(p)}\to\cZ_{M,p}$ is \emph{locally exact} (a map of polyhedral cones $f:Q\to P$ is called \emph{exact} when $(f^\gp)^{-1}(P)=Q$ \cite[Definition (4.6)\settowidth\dummy{$]$}]{katolog}, and it is called \emph{locally exact} when for every face $F\subseteq P$, the localized map $Q+f^{-1}(F)^\gp\to P+F^\gp$ is exact \cite[(A.3.2)(iii)\settowidth\dummy{$]$}]{illusiekatonakayama}\cite[Definition 2.1(3)\settowidth\dummy{$]$}]{nakayamaogus}).
Broken submersions model degenerating families of log smooth manifolds.

\begin{example}
The multiplication map $(x,y)\mapsto xy$ is a broken submersion $\pRR_{\geq 0}^2\to\pRR_{\geq 0}$.
We may add circle factors and consider the map $(S^1\times\pRR_{\geq 0})^2\to S^1\times\pRR_{\geq 0}$ given by $(\theta,\phi,x,y)\mapsto(\theta+\phi,xy)$.
This is an `oriented real blow-up' of the standard complex analytic nodal degeneration $\CC^2\to\CC$ given by $(z,w)\mapsto zw$ (take $z=xe^{i\theta}$ and $w=ye^{i\phi}$).
The present context of log smooth manifolds and broken submersions thereof keeps track of a `matching' of `circles at infinity' of the two sides of a node (in the present example, the circle $S^1\times 0$ in the base parameterizes all possible matchings, but this need not be the case in an arbitrary broken submersion).
\end{example}

It is a nontrivial result that broken submersions are preserved under pullback.
A \emph{simply-broken submersion} is a broken submersion which is locally (on the source) a pullback of the broken submersion $\pRR_{\geq 0}^2\times\RR^k\to\pRR_{\geq 0}$ given by $(x,y,p)\mapsto xy$.
A \emph{strict submersion} is locally modelled on a pullback of $\RR^k\to *$.

We may now formulate a precise expectation.
We consider a simply-broken submersion $C\to B$ of relative dimension two, with fiberwise complex structure, and we consider a strict submersion $W\to C$.
There is then a moduli stack $\uHol_B(C,W)_{\Log\Sm}$ on log smooth manifolds defined as in Section \ref{modulistackssection} (the definition there is purely `diagrammatic', hence does not care what category we work with, provided it has the relevant pullbacks and a notion of `vertically pseudo-holomorphic').
Following Joyce, we expect the regular locus $\uHol_B(C,W)_{\Log\Sm}^\reg$ inside such moduli stacks to be representable (Parker proves a closely related result \cite[Theorem 6.8]{parkerkuranishi}).
The essential analytic content in this result (beyond that already contained in the case of families of compact Riemann surfaces) is the exponential decay of pseudo-holomorphic maps in cylindrical coordinates (in all derivatives, including derivatives in the base direction, taken in log coordinates) as in Fukaya--Oh--Ohta--Ono \cite[A1.58]{foooII}.
While broken submersions are not locally trivial, they do always have connections (e.g.\ the everywhere non-vanishing vector field $\frac 12(x\partial_x+y\partial_y)$ on $\pRR_{\geq 0}^2$ can be taken as the horizontal distribution of a connection on the multiplication map $\pRR_{\geq 0}^2\to\pRR_{\geq 0}$), which allows for an implementation of the inductive strategy for proving smoothness of moduli spaces given that they are $C^1$, as discussed in the proof of Theorem \ref{representabilitytheorem}.

The language of log smooth manifolds can also be used to model degenerations of the target as considered in symplectic field theory (using the compactifications from either \cite{behwz} or \cite{pardoncontact}) and to describe the moduli spaces of `witch curves' used to construct $A_\infty$-functors associated to Lagrangian correspondences and their compositions \cite{bottmanassociahedra,bottmanwitch,abouzaidbottman}.

Finally, one could hope to show the entire moduli stack $\uHol_B(C,W)$ to be representable on a suitable $\infty$-category of `derived log smooth manifolds'.

\bibliographystyle{amsplain}
\bibliography{nonlinear}
\addcontentsline{toc}{section}{References}

\address{Simons Center for Geometry and Physics\\
State University of New York\\
Stony Brook, NY 11794-3636
U.S.A.}

\end{document}